\documentclass[reqno]{amsart}
\usepackage{mathrsfs}
\usepackage{amssymb}
\usepackage{amscd}
\theoremstyle{plain}
\newtheorem{theorem}{Theorem}[section]
\newtheorem{lemma}{Lemma}[section]
\newtheorem{proposition}{Proposition}[section]
\newtheorem{corollary}{Corollary}[section]

\newtheorem{conjecture}{Conjecture}[section]

\begin{document}
\title{Images of linear polynomials on upper triangular matrix algebras}
\thanks{$^*$ Corresponding author}
\maketitle
\begin{center}
Yingyu Luo\\
College of Mathematics, Changchun Normal University, Changchun
130032, China\\
E-mail: luoyingyu1980@163.com\\
Qian Chen$^*$\\
Department of Mathematics, Shanghai Normal University,
Shanghai 200234, China.\\
Email address: qianchen0505@163.com
\end{center}
\maketitle

\begin{abstract}
The Fagundes-Mello conjecture asserts that every multilinear polynomial on upper triangular matrix algebras is a vector space, which is an improtant variation of the old and famous Lvov-Kaplansky conjecture. The goal of the paper is to give a description of the images of linear polynomials with zero constant term on the upper triangular matrix algebra under a mild condition on the ground field. As a consequence we improve all results on the Fagundes-Mello conjecture. As another consequence we improve some results by Fagundes and Koshlukov on the images of multilinear graded polynomials on upper triangular matrix algebras.
\end{abstract}

\section{Introduction}
Let $n\geq 2$ be an integer. Let $K$ be a field and let $K\langle X\rangle$ be the free associative algebra over $K$, freely generated by the countable set $X=\{x_1,x_2,\ldots\}$ of
noncommtative variables. We refer to the elements of $K\langle X\rangle$ as
polynomials.

Images of polynomials evaluated on algebras play an important role in noncommutative
algebra. In particular, various challenging problems related to the theory of polynomial identities
have been settled after the construction of central polynomials on matrix algebras by Formanek \cite{Formanek} and Razmyslov \cite{Razmyslov}.

The old and famous Lvov-Kaplansky conjecture asserts:

\begin{conjecture}\cite{Do}\label{Con1}
Let $f$ be a multilinear polynomial. Then the set of values
of $f$ on the matrix algebra $M_n(K)$ over a field $K$ is a vector space.
\end{conjecture}

A special case on polynomials of degree two has been known for long
time (see \cite{Albert,Sho}). In 2013 Mesyan \cite{Mes} extended this result
for nonzero multilinear polynomials of degree three. In the same year, Buzinski and
Winstanley \cite{BW} extended this result for nonzero multilinear
polynomials of degree four.

In 2012 Kanel-Belov, Malev, and Rowen \cite{Rowen1} gave a complete
description of the image of semi-homogeneous polynomials on the
algebra of $2\times 2$ matrices over a quadratically closed field. As a consequence, they solved Conjecture \ref{Con1} for $n=2$. In 2016 they gave a complete description of the image of a multilinear polynomial which is trace vanishing
on $3\times 3$ matrices over a field $K$ (see \cite[Theorem 4]{Rowen2}).
We remark that Conjecture \ref{Con1} remains open for $n=3$.

In spite of many efforts, however, Conjecture \ref{Con1} seems
to be far from being resolved.  In 2020 Bre\v{s}ar \cite{Bre} presented some rough approximate versions of Conjecture \ref{Con1}, or
at least as an attempt to approach this conjecture from a different perspective.

In attempts to approach Conjecture \ref{Con1}, some
variations of it have been studied extensively. For example, the images of
multilinear polynomials of small degree on Lie Algebras \cite{An,SP}
and Jordan Algebras \cite{Ma, Malev3} have been discussed. In 2021 Malev \cite{Malev2} gave a complete description of the images of multilinear polynomials evaluated on the quaternion algebra.
In 2021 Vitas \cite{Vitas2} proved for any nonzero multilinear polynomial $p$,
that if $\mathcal{A}$ is an algebra with a surjective inner derivation, such as the Weyl algebra, then $p(\mathcal{A})=\mathcal{A}$. In 2022 Kanel-Belov, Malev, Pines, and Rowen \cite{Rowen2022} investigated the images of multilinear and semihomogeneous polynomials on the
algebra of octonions. Recently Centrone and Mello \cite{CM2023} investigated the images of graded polynomials on matrix algebras. For the most recent results on images of polynomials we recommend the survey paper \cite{survey}.

The set of all $n\times n$ upper
triangular matrices over $K$ will be denoted by $UT_n$. The set of all $n\times n$ strictly upper
triangular matrices will be denoted by $UT_n^{(0)}$. More generally, if $t\geq 0$, the set of all
upper triangular matrices whose entries $(i,j)$ are zero, for $j-i\leq t$, will be denoted by $UT_n^{(t)}$. For convenience we set $UT_n^{(-1)}=UT_n$. It is easy to check that $J=UT_n^{(0)}$ and $J^k=UT_n^{(k-1)}$ for all $k\geq 1$, where $J$ is the Jacobson radical of $UT_n$ (see \cite[Example 5.58]{Bre3}).

In 2019 Fagundes \cite{Fag} gave a complete description of the images
of multilinear polynomials on strictly upper trinagular matrix algebras over a field. In the same year, Fagundes and Mello \cite{FM} discussed the images
of multilinear polynomials of degree up to four on upper triangular matrix algebras. They proposed the following important variation of Conjecture \ref{Con1}:

\begin{conjecture}\cite[Conjecture 1]{FM}\label{Con}
The image of a multilinear polynomial over a field $K$ on $UT_n$ is always a vector space.
\end{conjecture}

In 2019 Wang \cite{Wang2019} gave a positive answer of Conjecture \ref{Con} for $n=2$ (see also \cite{Wang20191} for a correction of the paper). We remark that Fagundes gave a positive answer of Conjecture \ref{Con} for $n=2$ with an extremely simpler proof in his master's degree dissertation (see \cite{Fagm}), but the text was written in Portuguese and the result was not published elsewhere.

In 2021 Mello \cite{Mello1} gave a positive answer of Conjecture \ref{Con} for $n=3$ on an infinite field. In 2022 Gargate and Mello \cite{Mello} gave a positive answer of Conjecture \ref{Con} on an infinite field. In the same year, Luo and Wang \cite{Wang2022} gave a positive answer of Conjecture \ref{Con} under a mild condition on the ground field $K$. More precisely, they gave the following result:

\begin{theorem}\cite[Theorem 1.1]{Wang2022}\label{T1}
Let $K$ be a field, let $m\geq 1$ be an integer, let $n\geq 2$ be an integer. Let $p(x_1,\ldots,x_m)$ be a nonzero multilinear polynomial in non-commutative variables over $K$. Suppose that $|K|>\frac{n(n-1)}{2}$. We have that $p(UT_n)=UT_n^{(t)}$ for some integer $-1\leq t\leq \frac{m}{2}-1$.
\end{theorem}

As an application of Theorem \ref{T1}, Fagundes and Koshlukov \cite{FK} investigated the images of multilinear graded polynomials on upper triangular matrix algebras over a field.

In 2021 Wang, Zhou, and Luo \cite{WangZL} gave a complete description of
the images of polynomials with zero constant term on $2\times 2$ upper triangular matrix algebras over an algebraically closed field $K$. In 2022 Chen, Luo, and Wang \cite{CLW} gave a complete description of the images of polynomials with zero constant term on $3\times 3$ upper triangular matrix algebras over an algebraically closed field $K$. In the same year, Panja and Prasad \cite{PP} discussed some cases of the images of polynomials with zero constant term on upper triangular matrix algebras over an algebraically closed field $K$. Recently Chen \cite{ChenLuoWang1} gave a complete description of the images of polynomials with zero constant term on upper triangular matrix algebras over an algebraically closed field.

In the present paper we shall give a complete description of the images of linear polynomials with zero constant term on upper triangular matrix algebras under a mild condition on the ground field. More precisely, we shall prove the following result.

\begin{theorem}\label{T2}
Let $K$ be a field, let $m\geq 1$ be an integer, let $n\geq 2$ be an integer. Let $p(x_1,\ldots,x_m)$ be a nonzero linear polynomial in non-commutative variables over $K$. Suppose that $|K|>\frac{n(n-1)}{3}$. We have that $p(UT_n)=UT_n^{(t)}$ for some integer $-1\leq t\leq \frac{m}{2}-1$.
\end{theorem}

As a consequence of Theorem \ref{T2} we have the following result, which obviously improves Theorem \ref{T1}.

\begin{corollary}\label{CC}
Let $K$ be a field, let $m\geq 1$ be an integer, let $n\geq 2$ be an integer. Let $p(x_1,\ldots,x_m)$ be a nonzero multilinear polynomial in non-commutative variables over $K$. Suppose that $|K|>\frac{n(n-1)}{3}$. We have that $p(UT_n)=UT_n^{(t)}$ for some integer $-1\leq t\leq \frac{m}{2}-1$.
\end{corollary}

We organize this paper as follows: In Section $2$ we shall give some prelimiaries. In Section $3$ we shall give the proof of Theorem \ref{T2}. In Section $4$, as an application of Corollary \ref{CC} we shall improve two results obtained by Fagundes and Koshlukov in \cite{FK}.

We remark that the proof of Theorem \ref{T2} is completely different from that of Theorem \ref{T1}. We also remark that the proof of Theorem \ref{T2} comes from two aspects: one comes from some arguments in \cite{ChenLuoWang1}, another comes from some new arguments.

\section{prelimiaries}

By $\mathcal{N}$ we denote the set of all positive integers. Let $K$ be a field. For $n\in \mathcal{N}$ with $n\geq 2$, we can write
\[
UT_n=\left(
\begin{array}{cc}
UT_{n-1} & K^{n-1}\\
 & K
\end{array} \right).
\]

Let $\mathcal{A}$ be an algebra over $K$. Denote by $\mathcal{T}(\mathcal{A})$ the set of all polynomial identities of $\mathcal{A}$. Then it is easy to see that
\[
\mathcal{T}(K)\supset \mathcal{T}(UT_2)\supset \mathcal{T}(UT_3)\supset\cdots.
\]
So, given a polynomial $p(x_1,\ldots,x_m)$ in non-commutative variables over $K$, we define its \textbf{order}
as the least integer $m$ such that $p\in \mathcal{T}(UT_m)$ but $p\not\in \mathcal{T}(UT_{m+1})$.
Note that $T_1(K)=K$. A polynomial $p$ has order $0$ if $p\not\in \mathcal{T}(K)$. We denote the order of $p$ by $\mbox{ord}(p)$. For a detailed introduction of the order of polynomials
we refer the reader to the book \cite[Chapter 5]{Drensky}.

We remark that the structure of $\mathcal{T}(UT_n)$ is essentially known. For instance,
the basis of $\mathcal{T}(UT_n)$ has been described in \cite[Theorem 5.4]{Drensky},
if $K$ is infinite field. A similar insight holds in the context of finite fields.

Suppose that $p(x_1,\ldots,x_m)$ be a nonzero polynomial with zero constant term over $K$. Suppose that $\mbox{ord}(p)=r$, where $r\geq 2$. Since $p(UT_r)=\{0\}$ we easily check that $2r\leq m$.

Let $m,k\in \mathcal{N}$. We set
\[
T^k_m=\left\{(i_1,\ldots,i_k)\in \mathcal{N}^k~|~\mbox{$1\leq i_1,\ldots,i_k\leq m$, $i_s\neq i_t$ for all $s\neq t$}\right\}.
\]

Let $p(x_1,\ldots,x_m)$ be a linear polynomial with zero constant term over a field $K$. We can write
\begin{equation}\label{e1}
p(x_1,\ldots,x_m)=\sum\limits_{k=1}^m\left(\sum\limits_{(i_1,\ldots,i_k)\in T^k_m}\lambda_{i_1\cdots i_k}x_{i_1}\cdots x_{i_k}\right),
\end{equation}
where $\lambda_{i_1\cdots i_k}\in K$.

The following result is similar to \cite[Lemma 3.2]{ChenLuoWang1}. We give its proof for completeness.

\begin{proposition}\label{P1}
Let $p(x_1,\ldots,x_m)$ be a linear polynomial with zero constant term in
non-commutative variables over $K$. For $u_i=(a_{jk}^{(i)})\in UT_n$, $i=1,\ldots,m$, we set
\[
\bar{a}_{jj}=(a_{jj}^{(1)},\ldots,a_{jj}^{(m)})
\]
for all $j=1,\ldots,n$. We have that
\begin{equation}\label{e2}
p(u_1,\ldots,u_m)=(p_{st}),
\end{equation}
where $p_{ss}=p(\bar{a}_{ss})$ for all $s=1,\ldots,n$, and
\[
p_{st}=\sum\limits_{k=1}^{t-s}\left(\sum\limits_{\substack{s=j_1<j_2<\cdots <j_{k+1}=t\\(i_1,\ldots,i_k)\in T^k_m}}
p_{i_1\cdots i_k}(\bar{a}_{j_1j_1},\ldots,\bar{a}_{j_{k+1}j_{k+1}})a_{j_1j_2}^{(i_1)}\cdots a_{j_{k}j_{k+1}}^{(i_k)}\right)
\]
for all $1\leq s<t\leq n$, where $p_{i_1,\ldots,i_k}(\bar{z}_1,\ldots,\bar{z}_{k+1})$ is a polynomial on $m(k+1)$-variables over $K$, where $k=1,\ldots,n-1$.
\end{proposition}

\begin{proof}
For any $(i_1,\ldots,i_k)\in T^k_m$, $k=1,\ldots,m$, we easily check that
\[
u_{i_1}\cdots u_{i_k}=(m_{st}),
\]
where
\[
m_{st}=\sum\limits_{s=j_1\leq j_2\leq \cdots \leq j_{k+1}=t}a_{j_1j_2}^{(i_1)}\cdots a_{j_{k}j_{k+1}}^{(i_k)}
\]
for all $1\leq s\leq t\leq n$. It follows from (\ref{e1}) that

\begin{eqnarray*}
\begin{split}
p(u_1,\ldots,u_m)&=\sum\limits_{k=1}^m\left(\sum\limits_{(i_1,\ldots,i_k)\in T^k_m}\lambda_{i_1\cdots i_k}u_{i_1}\cdots u_{i_k}\right)\\
&=\sum\limits_{k=1}^m\left(\sum\limits_{(i_1,\ldots,i_k)\in T^k_m}\lambda_{i_1\cdots i_k}\left(
\begin{array}{cccc}
m_{11} & m_{12} & \ldots & m_{1n}\\
0 & m_{22} & \ldots & m_{2n}\\
 \vdots & \vdots & \ddots & \vdots\\
 0 & 0 & \ldots & m_{nn}
\end{array} \right)\right)\\
&=(p_{st}),
\end{split}
\end{eqnarray*}
where
\begin{eqnarray*}
\begin{split}
p_{st}&=\sum\limits_{k=1}^m\left(\sum\limits_{(i_1,\ldots,i_k)\in T^k_m}\lambda_{i_1\cdots i_k}m_{st}\right)\\
&=\sum\limits_{k=1}^m\left(\sum\limits_{(i_1,\ldots,i_k)\in T^k_m}\lambda_{i_1\cdots i_k}\left(
\sum\limits_{s=j_1\leq j_2\leq \cdots \leq j_{k+1}=t}a_{j_1j_2}^{(i_1)}\cdots a_{j_{k}j_{k+1}}^{(i_k)}\right)\right)\\
&=\sum\limits_{k=1}^{m}\left(\sum\limits_{\substack{s=j_1\leq j_2\leq \cdots \leq j_{k+1}=t\\(i_1,\ldots,i_k)\in T^k_m}}\lambda_{i_1\cdots i_k}a_{j_1j_2}^{(i_1)}\cdots a_{j_{k}j_{k+1}}^{(i_k)}\right),
\end{split}
\end{eqnarray*}
where $1\leq s\leq t\leq n$. We get that
\begin{eqnarray*}
\begin{split}
p_{ss}&=\sum\limits_{k=1}^{m}\left(\sum\limits_{(i_1,\ldots,i_k)\in T^k_m}\lambda_{i_1\cdots i_k}a_{ss}^{(i_1)}\cdots a_{ss}^{(i_k)}\right)\\
&=p(\bar{a}_{ss}),
\end{split}
\end{eqnarray*}
for all $s=1,\ldots,n$, and
\begin{eqnarray*}
\begin{split}
p_{st}&=\sum\limits_{k=1}^{m}\left(\sum\limits_{\substack{s=j_1\leq j_2\leq \cdots \leq j_{k+1}=t\\(i_1,\ldots,i_k)\in T^k_m}}\lambda_{i_1\cdots i_k}a_{j_1j_2}^{(i_1)}\cdots a_{j_{k}j_{k+1}}^{(i_k)}\right)\\
&=\sum\limits_{k=1}^{t-s}\left(\sum\limits_{\substack{s=j_1<j_2<\cdots <j_{k+1}=t\\(i_1,\ldots,i_k)\in T^k_m}}
p_{i_1\cdots i_k}(\bar{a}_{j_1j_1},\ldots,\bar{a}_{j_{k+1}j_{k+1}})a_{j_1j_2}^{(i_1)}\cdots a_{j_{k}j_{k+1}}^{(i_k)}\right)
\end{split}
\end{eqnarray*}
for all $1\leq s<t\leq n$, where $p_{i_1,\ldots,i_k}(\bar{z}_1,\ldots,\bar{z}_{k+1})$ is a polynomial on $m(k+1)$-variables over $K$, where $k=1,\ldots,n-1$. This proves the result.
\end{proof}

\section{the proof of Theorem \ref{T2}}

Let $p(x_1,\ldots,x_m)$ be a linear polynomial with zero constant term over $K$. For any $(i_1,\ldots,i_k)\in T^k_m$, $k=1,\ldots,m$, by $\alpha_{i_1\cdots i_k}$ we denote the sum of coefficients of all monomials consisting of $x_{i_1},\ldots,x_{i_k}$ in $p(x_1,\ldots,x_m)$.

We begin with the following result, which is of some independent interests.

\begin{lemma}\label{L3.1}
Let $m\geq 1$ be integer. Let $K$ be a field. Let $p(x_1,\ldots, x_m)$ be a linear polynomial with zero constant term over $K$. Suppose that $p(K)\neq\{0\}$. We have that $p(K)=K$.
\end{lemma}

\begin{proof}
For any $a_1,\ldots,a_m\in K$ we get from (\ref{e1}) that
\begin{eqnarray}\label{e33}
\begin{split}
p(a_1,\ldots,a_m)&=\sum\limits_{k=1}^m\left(\sum\limits_{(i_1,\ldots,i_k)\in T^k_m}\lambda_{i_1\cdots i_k}a_{i_1}\cdots a_{i_k}\right)\\
&=\sum\limits_{k=1}^m\left(\sum\limits_{1\leq i_1<\cdots <i_k\leq m}\alpha_{i_1\cdots i_k}a_{i_1}\cdots a_{i_k}\right).
\end{split}
\end{eqnarray}
Since $p(K)\neq\{0\}$, we have that there exists a minimum integer $r\geq 1$ such that
\[
\alpha_{i_1'\cdots i_r'}\neq 0
\]
for some $1\leq i_1'<\cdots <i_r'\leq m$. We rewrite (\ref{e33}) as follows:
\begin{equation}\label{e44}
p(a_1,\ldots,a_m)=\sum\limits_{k=r}^m\left(\sum\limits_{1\leq i_1<\cdots <i_k\leq m}\alpha_{i_1\cdots i_k}a_{i_1}\cdots a_{i_k}\right).
\end{equation}

For any $a'\in K$, we take $a_{i}\in K$, $i=1,\ldots,m$, where
\[
\left\{
\begin{aligned}
a_{i_1'}&=\alpha_{i_1'\cdots i_r'}^{-1}a';\\
a_{i_k'}&=1_{K},\quad\mbox{$k=2,\ldots,r$};\\
a_i&=0, \quad\mbox{for all $i\not\in \{i_1',\ldots,i_r'\}$}.
\end{aligned}
\right.
\]
It follows from (\ref{e44}) that
\begin{eqnarray*}
\begin{split}
p(a_1,\ldots,a_m)&=\alpha_{i_1'\cdots i_r'}a_{i_1'}\cdots a_{i_r'}\\
&=\alpha_{i_1'\cdots i_r'}\alpha_{i_1'\cdots i_r'}^{-1}a'\\
&=a'.
\end{split}
\end{eqnarray*}
This implies that $p(K)=K$. This proves the result.
\end{proof}

Using the same argument as that of Lemma \ref{L3.1} we can obtain the following result.

\begin{lemma}\label{L3.2}
Let $n\geq 2$ and $m\geq 1$ be integers. Let $K$ be a field. Let $p(x_1,\ldots, x_m)$ be a linear polynomial with zero constant term over $K$. Suppose that $p(K)\neq\{0\}$. We have that $p(UT_n)=UT_n$.
\end{lemma}

\begin{proof}
For any $u_1,\ldots,u_m\in UT_n$ with $u_iu_j=u_ju_i$ for all $i,j=1,\ldots,m$, we get from (\ref{e1}) that
\begin{eqnarray}\label{e3}
\begin{split}
p(u_1,\ldots,u_m)&=\sum\limits_{k=1}^m\left(\sum\limits_{(i_1,\ldots,i_k)\in T^k_m}\lambda_{i_1\cdots i_k}u_{i_1}\cdots u_{i_k}\right)\\
&=\sum\limits_{k=1}^m\left(\sum\limits_{1\leq i_1<\cdots <i_k\leq m}\alpha_{i_1\cdots i_k}u_{i_1}\cdots u_{i_k}\right).
\end{split}
\end{eqnarray}
Since $p(K)\neq\{0\}$, we have that there exists a minimum integer $r\geq 1$ such that
\[
\alpha_{i_1'\cdots i_r'}\neq 0
\]
for some $1\leq i_1'<\cdots <i_r'\leq m$. We rewrite (\ref{e3}) as follows:
\begin{equation}\label{eeee4}
p(u_1,\ldots,u_m)=\sum\limits_{k=r}^m\left(\sum\limits_{1\leq i_1<\cdots <i_k\leq m}\alpha_{i_1\cdots i_k}u_{i_1}\cdots u_{i_k}\right).
\end{equation}

For any $u'\in UT_n$, we take $u_{i}\in UT_n$, $i=1,\ldots,m$, where
\[
\left\{
\begin{aligned}
u_{i_1'}&=\alpha_{i_1'\cdots i_r'}^{-1}u';\\
u_{i_k'}&=1_{UT_n},\quad\mbox{$k=2,\ldots,r$};\\
u_i&=0, \quad\mbox{for all $i\not\in \{i_1',\ldots,i_r'\}$}.
\end{aligned}
\right.
\]
Note that $u_iu_j=u_ju_i$ for all $i,j=1,\ldots,m$. It follows from (\ref{eeee4}) that
\begin{eqnarray*}
\begin{split}
p(u_1,\ldots,u_m)&=\alpha_{i_1'\cdots i_r'}u_{i_1'}\cdots u_{i_r'}\\
&=\alpha_{i_1'\cdots i_r'}\alpha_{i_1'\cdots i_r'}^{-1}u'\\
&=u'.
\end{split}
\end{eqnarray*}
This implies that $p(UT_n)=UT_n$. This proves the result.
\end{proof}

The following result is similar to \cite[Lemma 3.6]{ChenLuoWang1}. We give its proof for completeness.

\begin{lemma}\label{L3.3}
Let $p(x_1,\ldots, x_m)$ be a linear polynomial with zero constant term over $K$. Let $p_{i_1,\ldots,i_k}(\bar{z}_1,\ldots,\bar{z}_{k+1})$ is a polynomial on $m(k+1)$-variables over $K$ in (\ref{e2}), where $k=1,\ldots,n-1$. For any $1\leq r\leq n-1$, we have that ord$(p)=r$ if and only if
\begin{enumerate}
\item[(i)] $p(K)=\{0\}$;
\item[(ii)] $p_{i_1\cdots i_s}(K)=\{0\}$ for all $(i_1,\ldots,i_s)\in T^s_m$, where $s=1,\ldots,r-1$ and $r\geq 2$;
\item[(iii)] $p_{i_1'\cdots i_r'}(K)\neq \{0\}$ for some $(i_1',\ldots,i_r')\in T^r_m$.
\end{enumerate}
\end{lemma}

\begin{proof}
Suppose that the statements (i), (ii), and (iii) hold true. We claim that ord$(p)=r$.
It follows from the statement (i) that ord$(p)\geq 1$.

In view of the statement (iii) we have that
\[
p_{i_1',\ldots,i_r'}(K)\neq\{0\}
\]
for some $(i_1',\ldots,i_r')\in T_r$. We have that there exist $\bar{b}_{s}\in K^m$, $s=1,\ldots,r+1$, such that
\[
p_{i_1',\ldots,i_r'}(\bar{b}_{1},\ldots,\bar{b}_{r+1})\neq 0.
\]
We take $u_i=(a_{jk}^{(i)})\in T_{r+1}(K)$, $i=1,\ldots,m$, where
\[
\left\{
\begin{aligned}
\bar{a}_{ss}&=\bar{b}_{s},\quad\mbox{for all $s=1,\ldots,r+1$};\\
a_{s,s+1}^{(i_s')}&=1,\quad\mbox{for all $s=1,\ldots,r$};\\
a_{jk}^{(i)}&=0,\quad\mbox{otherwise}.
\end{aligned}
\right.
\]
It follows from the both the statement (ii) and (\ref{e2}) that
\begin{eqnarray*}
\begin{split}
p_{1,r+1}&=\sum\limits_{\substack{1=j_1<j_2<\cdots <j_{r+1}=r+1\\(i_1,\ldots,i_r)\in T^r_m}}
p_{i_1\cdots i_r}(\bar{a}_{j_1j_1},\ldots,\bar{a}_{j_{r+1}j_{r+1}})a_{j_1j_2}^{(i_1)}\cdots a_{j_{r}j_{r+1}}^{(i_r)}\\
&=p_{i_1',\ldots,i_r'}(\bar{b}_{1},\ldots,\bar{b}_{r+1})a_{12}^{(i_1')}\cdots a_{r,r+1}^{(i_r')}\\
&=p_{i_1',\ldots,i_r'}(\bar{b}_{1},\ldots,\bar{b}_{r+1})\neq 0.
\end{split}
\end{eqnarray*}
This implies that $p(T_{r+1}(K))\neq \{0\}$. For any $u_{i}=(a_{jk}^{(i)})\in T_r(K)$, $i=1,\ldots,m$, we get from both the statement (i), the statement (ii), and (\ref{e2}) that
\begin{equation}\label{e4}
p(u_1,\ldots,u_m)=(p_{st}),
\end{equation}
where $p_{ss}=0$ for all $s=1,\ldots,n$, and
\begin{eqnarray*}
\begin{split}
p_{st}&=\sum\limits_{k=1}^{t-s}\left(\sum\limits_{\substack{s=j_1<j_2<\cdots <j_{k+1}=t\\(i_1,\ldots,i_k)\in T^k_m}}
p_{i_1i_2\cdots i_k}(\bar{a}_{j_1j_1},\ldots,\bar{a}_{j_{k+1}j_{k+1}})a_{j_1j_2}^{(i_1)}\cdots a_{j_{k}j_{k+1}}^{(i_k)}\right)\\
&=0
\end{split}
\end{eqnarray*}
for all $1\leq s<t\leq r$. This implies that $p(T_r(K))=\{0\}$. We obtain that ord$(p)=r$.

Suppose that ord$(p)=r$. We claim that the statements (i), (ii), and (iii) hold true. Since $r\geq 1$ we get that the statement (i) holds true. We now claim that the statement (ii) holds true. Suppose on the contrary that
\[
p_{i_1',\ldots,i_s'}(K)\neq \{0\}
\]
for some $(i_1',\ldots,i_s')\in T^s_m$, where $1\leq s\leq r-1$. Then there exist $\bar{b}_{j}\in K^m$, where $j=1,\ldots,s+1$ such that
\[
p_{i_1',\ldots,i_s'}(\bar{b}_{1},\ldots,\bar{b}_{s+1})\neq 0.
\]
We take $u_i=(a_{jk}^{(i)})\in T_{s+1}(K)$, $i=1,\ldots,m$, where
\[
\left\{
\begin{aligned}
\bar{a}_{tt}&=\bar{b}_{t},\quad\mbox{for all $t=1,\ldots,s+1$};\\
a_{t,t+1}^{(i_t')}&=1,\quad\mbox{for all $t=1,\ldots,s$};\\
a_{jk}^{(i)}&=0,\quad\mbox{otherwise}.
\end{aligned}
\right.
\]
We get from (\ref{e2}) that
\begin{eqnarray*}
\begin{split}
p_{1,s+1}&=\sum\limits_{k=1}^{s}\left(\sum\limits_{\substack{1=j_1<j_2<\cdots <j_{k+1}=s+1\\(i_1,\ldots,i_k)\in T^k_m}}
p_{i_1\cdots i_k}(\bar{a}_{j_1j_1},\ldots,\bar{a}_{j_{k+1}j_{k+1}})a_{j_1j_2}^{(i_1)}\cdots a_{j_{k}j_{k+1}}^{(i_k)}\right)\\
&=p_{i_1',i_2',\ldots,i_s'}(\bar{b}_{1},\ldots,\bar{b}_{s+1})\neq 0.
\end{split}
\end{eqnarray*}
This implies that $p(T_{s+1}(K))\neq\{0\}$, a contradiction. This proves the statement (ii).

We finally claim that the statement (iii) holds true. Note that $p(T_{1+r}(K))\neq\{0\}$. Thus,
we have that there exist $u_{i}=(a_{jk}^{(i)})\in T_{1+r}(K)$, $i=1,\ldots,m$, such that
\[
p(u_1,\ldots,u_m)=(p_{st})\neq 0.
\]
In view of the statement (ii) we get that
\begin{eqnarray*}
\begin{split}
p_{1,r+1}&=\sum\limits_{\substack{1=j_1<j_2< \cdots <j_{r+1}=r+1\\(i_1,\ldots,i_r)\in T^r_m}}p_{i_1i_2\cdots i_r}(\bar{a}_{j_1j_1},\ldots,\bar{a}_{j_{r+1}j_{r+1}})a_{j_1j_2}^{(i_1)}\cdots a_{j_{r}j_{r+1}}^{(i_r)}\\
&\neq 0.
\end{split}
\end{eqnarray*}
This implies that $p_{i_1',\ldots,i_r'}(K)\neq \{0\}$ for some $(i_1',\ldots,i_r')\in T^r_m$. This proves the statement (iii). The proof of the result is now complete.
\end{proof}

The following result is crucial for the proof of our main result, which is of independent interests.

\begin{lemma}\label{L3.4}
Let $r,s\geq 1$ be an integer. Let $K$ be a field. Let $U_i=\{i_1,\ldots,i_r\}$ be a subset of $\mathcal{N}$, $i=1,\ldots,s$. Let  $f_i(x_{i_1},\ldots,x_{i_r})$ be a linear polynomial over $K$ such that $f_i(K)\neq\{0\}$, $i=1,\ldots,s$. For any $u\in\bigcup_{i=1}^sU_i$ we set
\[
l(u)=\{i\in\{1,\ldots,s\}~|~u\in U_i\}.
\]
Suppose that
\[
|K|>max_{u\in \bigcup_{i=1}^sU_i}\{|l(u)|\}.
\]
We have that there exist $c_u\in K$, $u\in \bigcup_{i=1}^{s}U_i$, such that
\[
f_i(c_{i_1},\ldots,c_{i_{r}})\neq 0
\]
for all $i=1,\ldots,s$.
\end{lemma}

\begin{proof}
Without loss of generality we may assume that
\[
\{1,\ldots,n\}=\bigcup_{i=1}^sU_i.
\]
We set
\[
g_i(x_1,\ldots,x_n)=f_i(x_{i_1},\ldots,x_{i_r})
\]
for all $i=1,\ldots,s$. It is clear that $g_i$ is a linear polynomial over $K$ such that $g_i(K)\neq\{0\}$, $i=1,\ldots,s$.

For any $i\in\{1,\ldots,s\}$, since $g_i(K)\neq\{0\}$ we have that there exist $c_{1}^{(i)},\ldots,c_{n}^{(i)}\in K$ such that
\begin{equation}\label{eee1}
g_i(c_{1}^{(i)},\ldots,c_{n}^{(i)})\neq 0.
\end{equation}
We set
\[
h_{i}^{(1)}(x_1)=g_i(x_1,c_{2}^{(i)},\ldots,c_{n}^{(i)})
\]
for all $i=1,\ldots,s$. It follows from (\ref{eee1}) that
\[
h_{i}^{(1)}(c_{1}^{(i)})\neq 0
\]
for all $i=1,\ldots,s$. We see that $h_{i}^{(1)}(x_1)\neq 0$ for all $i=1,\ldots,s$. Since $|K|>|l(1)|$ and $h_i^{(1)}(x_1)$ is a nonzero linear polynomial, $i=1,\ldots,s$, we have that there exists $c_1\in K$ such that
\[
h_{i}^{(1)}(c_1)\neq 0
\]
for all $i=1,\ldots,s$. That is
\begin{equation}\label{eee2}
g_i(c_1,c_{2}^{(i)},\ldots,c_{n}^{(i)})\neq 0
\end{equation}
for all $i=1,\ldots,s$. Similarly, we set
\[
h_{i}^{(2)}(x_2)=g_i(c_1,x_2,c_{3}^{(i)},\ldots,c_{n}^{(i)})
\]
for all $i=1,\ldots,s$. It follows from (\ref{eee2}) that
\[
h_{i}^{(2)}(c_{2}^{(i)})\neq 0
\]
for all $i=1,\ldots,s$. We get that $h_i^{(2)}(x)\neq 0$ for all $i=1,\ldots,s$. Since $|K|>|l(2)|$ and $h_{i}^{(2)}(x_2)$ is a nonzero linear polynomial, $i=1,\ldots,s$, we have that there exists $c_2\in K$ such that
\[
h_{i}^{(2)}(c_2)\neq 0
\]
for all $i=1,\ldots,s$. That is
\[
g_i(c_1,c_{2},c_{3}^{(i)},\ldots,c_{n}^{(i)})\neq 0
\]
for all $i=1,\ldots,s$. Continuing the same arguments as above we can obtain that there exist $c_3,\ldots,c_n\in K$ such that
\[
g_i(c_1,\ldots,c_n)\neq 0
\]
for all $i=1,\ldots,s$.  This implies that there exist $c_u\in K$, $u\in \bigcup_{i=1}^{s}U_i$, such that
\[
f_i(c_{i_1},\ldots,c_{i_{r}})\neq 0
\]
for all $i=1,\ldots,s$. This proves the result.
\end{proof}

In view of Lemma \ref{L3.4} we can obtain the following crucial result for the proof of our main result.

\begin{lemma}\label{L3.5}
Let $n\geq 2$ be an integer, let $1\leq r\leq n-1$ be an integer. Let $K$ be a field. Let
\[
f_{s,r+s+t}(\bar{x}_s,\bar{x}_{s+1},\ldots,\bar{x}_{r+s-1},\bar{x}_{r+s+t})
\]
be a nonzero linear polynomial with $m(r+1)$-variables over $K$ such that
\[
f_{s,r+s+t}(K)\neq \{0\}
\]
for all $1\leq s<r+s+t\leq n$. Suppose that $|K|>\frac{(2n-3r+1)r}{2}$. We have that there exist $\bar{c}_i\in K^m$, $i=1,\ldots,n$, such that
\[
f_{s,r+s+t}(\bar{c}_s,\bar{c}_{s+1},\ldots,\bar{c}_{r+s-1},\bar{c}_{r+s+t})\neq 0
\]
for all $1\leq s<r+s+t\leq n$.
\end{lemma}

\begin{proof}
Note that
\[
\{\bar{x}_1,\ldots,\bar{x}_n\}=\bigcup_{1\leq s<r+s+t\leq n}\{\bar{x}_s,\ldots,\bar{x}_{r+s-1},\bar{x}_{r+s+t}\}.
\]

Set
\[
U_{s,r+s+t}=\{\bar{x}_s,\ldots,\bar{x}_{r+s-1},\bar{x}_{r+s+t}\}
\]
for all $1\leq s<r+s+t\leq n$. For any $i\in\{1,\ldots,n\}$ we set
\[
l(\bar{x}_i)=\{(s,r+s+t), 1\leq s<r+s+t\leq n~|~\bar{x}_i\in U_{s,r+s+t}\}.
\]
Note that
\[
\left\{
\begin{aligned}
f_{1,r+1}&(\bar{x}_1,\bar{x}_2,\ldots,\bar{x}_{r},\bar{x}_{r+1})\neq 0;\\
f_{1,r+2}&(\bar{x}_1,\bar{x}_2,\ldots,\bar{x}_{r},\bar{x}_{r+2})\neq 0;\\
&\vdots\\
f_{1,n}&(\bar{x}_1,\bar{x}_2,\ldots,\bar{x}_{r},\bar{x}_{n})\neq 0.
\end{aligned}
\right.
\]
We get that
\[
|l(\bar{x}_1)|\leq n-r.
\]
Note that
\[
\left\{
\begin{aligned}
f_{1,r+1}&(\bar{x}_1,\bar{x}_2,\ldots,\bar{x}_{r},\bar{x}_{r+1})\neq 0;\\
f_{1,r+2}&(\bar{x}_1,\bar{x}_2,\ldots,\bar{x}_{r},\bar{x}_{r+2})\neq 0;\\
&\vdots\\
f_{1,n}&(\bar{x}_1,\bar{x}_2,\ldots,\bar{x}_{r},\bar{x}_{n})\neq 0;\\
f_{2,r+2}&(\bar{x}_2,\bar{x}_3,\ldots,\bar{x}_{r},\bar{x}_{r+2})\neq 0;\\
f_{2,r+3}&(\bar{x}_2,\bar{x}_3,\ldots,\bar{x}_{r},\bar{x}_{r+3})\neq 0;\\
&\vdots\\
f_{2,n}&(\bar{x}_2,\bar{x}_3,\ldots,\bar{x}_{r},\bar{x}_{n})\neq 0.
\end{aligned}
\right.
\]
We get that
\[
|l(\bar{x}_2)|\leq (n-r)+(n-r-1).
\]
Note that
\[
\left\{
\begin{aligned}
f_{1,r+1}&(\bar{x}_1,\bar{x}_2,\ldots,\bar{x}_{r},\bar{x}_{r+1})\neq 0;\\
f_{1,r+2}&(\bar{x}_1,\bar{x}_2,\ldots,\bar{x}_{r},\bar{x}_{r+2})\neq 0;\\
&\vdots\\
f_{1,n}&(\bar{x}_1,\bar{x}_2,\ldots,\bar{x}_{r},\bar{x}_{n})\neq 0;\\
&\vdots\\
f_{r,2r}&(\bar{x}_r,\bar{x}_{r+1},\ldots,\bar{x}_{2r-1},\bar{x}_{2r})\neq 0;\\
f_{r,2r+1}&(\bar{x}_r,\bar{x}_{r+1},\ldots,\bar{x}_{2r-1},\bar{x}_{2r+1})\neq 0;\\
&\vdots\\
f_{r,n}&(\bar{x}_r,\bar{x}_{r+1},\ldots,\bar{x}_{2r-1},\bar{x}_{n})\neq 0.
\end{aligned}
\right.
\]
Continuing the same argument as above we have that
\[
|l(\bar{x}_i)|\leq \sum\limits_{j=1}^i(n-r-j+1)
\]
for all $i=1,\ldots,r$. Note that
\[
\sum\limits_{j=1}^r(n-r-j+1)=\frac{(2n-3r+1)r}{2}.
\]
This implies that
\begin{equation}\label{ee1}
|l(\bar{x}_i)|\leq \frac{(2n-3r+1)r}{2}
\end{equation}
for all $i=1,\ldots,r$. Note that
\[
\left\{
\begin{aligned}
f_{1,r+1}&(\bar{x}_1,\bar{x}_2,\ldots,\bar{x}_{r},\bar{x}_{r+1})\neq 0;\\
f_{1,r+2}&(\bar{x}_2,\bar{x}_3,\ldots,\bar{x}_{r+1},\bar{x}_{r+2})\neq 0;\\
&\vdots\\
f_{1,n}&(\bar{x}_2,\bar{x}_{3},\ldots,\bar{x}_{r+1},\bar{x}_{n})\neq 0;\\
&\vdots\\
f_{r,2r}&(\bar{x}_r,\bar{x}_{r+1},\ldots,\bar{x}_{2r-1},\bar{x}_{2r})\neq 0;\\
&\vdots\\
f_{r,n}&(\bar{x}_{r},\bar{x}_{r+1},\ldots,\bar{x}_{2r-1},\bar{x}_{n})\neq 0;\\
f_{r+1,2r+1}&(\bar{x}_{r+1},\bar{x}_{r+2},\ldots,\bar{x}_{2r},\bar{x}_{2r+1})\neq 0;\\
&\vdots\\
f_{r+1,n}&(\bar{x}_{r+1},\bar{x}_{r+2},\ldots,\bar{x}_{2r},\bar{x}_{n})\neq 0.
\end{aligned}
\right.
\]
We get that
\[
|l(\bar{x}_{r+1})|\leq 1+\sum\limits_{i=1}^{r}(n-r-i).
\]
Note that
\[
1+\sum\limits_{i=1}^{r}(n-r-i)=\frac{(2n-3r+1)r}{2}-r+1.
\]
This implies that
\[
|l(\bar{x}_{r+1})|\leq \frac{(2n-3r+1)r}{2}.
\]
Note that
\[
\left\{
\begin{aligned}
f_{1,r+1}&(\bar{x}_1,\bar{x}_2,\ldots,\bar{x}_{r},\bar{x}_{r+2})\neq 0;\\
f_{2,r+2}&(\bar{x}_2,\bar{x}_3,\ldots,\bar{x}_{r+1},\bar{x}_{r+2})\neq 0;\\
f_{3,r+3}&(\bar{x}_3,\bar{x}_4,\ldots,\bar{x}_{r+2},\bar{x}_{r+3})\neq 0;\\
&\vdots\\
f_{3,n}&(\bar{x}_3,\bar{x}_{4},\ldots,\bar{x}_{r+2},\bar{x}_{n})\neq 0;\\
&\vdots\\
f_{r,2r}&(\bar{x}_r,\bar{x}_{r+1},\ldots,\bar{x}_{2r-1},\bar{x}_{2r})\neq 0;\\
&\vdots\\
f_{r,n}&(\bar{x}_{r},\bar{x}_{r+1},\ldots,\bar{x}_{2r-1},\bar{x}_{n})\neq 0;\\
f_{r+1,2r+1}&(\bar{x}_{r+1},\bar{x}_{r+2},\ldots,\bar{x}_{2r},\bar{x}_{2r+1})\neq 0;\\
&\vdots\\
f_{r+1,n}&(\bar{x}_{r+1},\bar{x}_{r+2},\ldots,\bar{x}_{2r},\bar{x}_{n})\neq 0;\\
f_{r+2,2r+2}&(\bar{x}_{r+2},\bar{x}_{r+3},\ldots,\bar{x}_{2r+1},\bar{x}_{2r+2})\neq 0;\\
&\vdots\\
f_{r+2,n}&(\bar{x}_{r+2},\bar{x}_{r+3},\ldots,\bar{x}_{2r+1},\bar{x}_{n})\neq 0.
\end{aligned}
\right.
\]
We get that
\[
|l(\bar{x}_{r+2})|\leq 2+\sum\limits_{i=2}^{r+1}(n-r-i).
\]
Note that
\[
2+\sum\limits_{i=2}^{r+1}(n-r-i)\leq \frac{(2n-3r+1)r}{2}-2r+2.
\]
We get that
\[
|l(\bar{x}_{r+2})|\leq \frac{(2n-3r+1)r}{2}.
\]
Continuing the same argument as above we can obtain that
\begin{equation}\label{ee2}
|l(\bar{x}_{r+i})|\leq \frac{(2n-3r+1)r}{2}
\end{equation}
for all $i=1,\ldots,n-r$. It follows from both (\ref{ee1}) and (\ref{ee2}) that
\[
max_{1\leq i\leq n}\{|l(\bar{x}_i)|\}\leq \frac{(2n-3r+1)r}{2}.
\]
Since $|K|>\frac{(2n-3r+1)r}{2}$ we get that
\[
|K|>max_{1\leq i\leq n}\{|l(\bar{x}_i)|\}.
\]
In view of Lemma \ref{L3.4} we get that there exist $\bar{c}_i\in K^m$, $i=1,\ldots,n$, such that
\[
f_{s,r+s+t}(\bar{c}_s,\bar{c}_{s+1},\ldots,\bar{c}_{r+s-1},\bar{c}_{r+s+t})\neq 0
\]
for all $1\leq s<r+s+t\leq n$. This proves the result.
\end{proof}

As a consequence of Lemma \ref{L3.5} we have the following useful result.

\begin{corollary}\label{CC1}
Let $n\geq 2$ be an integer, let $2\leq r\leq n-1$ be an integer. Let $K$ be a field. Let $g_{s,r+s+t}(\bar{x}_{s,s+1},\ldots,\bar{x}_{r+s-2,r+s-1})$ be a nonzero linear polynomial with $m(r-1)$-variables over $K$ such that
\[
g_{s,r+s+t}(K)\neq \{0\}
\]
for all $1\leq s<r+s+t\leq n$. Suppose that $|K|>\frac{(2n-3r+1)r}{2}$. We have that there exist $\bar{c}_{i,i+1}\in K^m$, $i=1,\ldots,n-2$, such that
\[
g_{s,r+s+t}(\bar{c}_{s,s+1},\bar{c}_{s+1,s+2},\ldots,\bar{c}_{r+s-2,r+s-1})\neq 0
\]
for all $1\leq s<r+s+t\leq n$.
\end{corollary}

\begin{proof}
Set
\[
\bar{y}_s=\bar{x}_{s,s+1}
\]
for all $s=1,\ldots,n-2$. We set
\[
f_{s,r+s+t}(\bar{y}_{s},\ldots,\bar{y}_{r+s-1},\bar{y}_{r+s+t})=g_{s,r+s+t}(\bar{y}_{s},\ldots,\bar{y}_{r+s-2})
\]
for all $1\leq s<r+s+t\leq n$. Note that
\[
f_{s,r+s+t}(K)\neq\{0\}
\]
for all $1\leq s<r+s+t\leq n$. Note that
\[
\{1,\ldots,n\}=\bigcup_{1\leq s\leq r+s+t\leq n}\{s,\ldots,r+s-1,r+s+t\}.
\]
Since $|K|>\frac{(2n-3r+1)r}{2}$ we get from Lemma \ref{L3.5} that there exist $\bar{c}_i\in K^m$, $i=1,\ldots, n$, such that
\[
f_{s,r+s+t}(\bar{c}_{s},\ldots,\bar{c}_{r+s-1},\bar{c}_{r+s+t})\neq 0
\]
for all $1\leq s<r+s+t\leq n$.  We set
\[
\bar{c}_{i,i+1}=\bar{c}_i
\]
for all $i=1,\ldots,n-2$. This implies that
\[
g_{s,r+s+t}(\bar{c}_{s,s+1},\bar{c}_{s+1,s+2},\ldots,\bar{c}_{r+s-2,r+s-1})\neq 0
\]
for all $1\leq s<r+s+t\leq n$.
\end{proof}

The following result is similar to \cite[Lemma 3.10]{ChenLuoWang1}. We give its proof for completeness.

\begin{lemma}\label{L3.6}
Let $p(x_1,\ldots,x_m)$ be a linear
 polynomial with zero constant term over a field $K$. Suppose that ord$(p)=1$ and $|K|\geq n$. We have that $p(UT_n)=UT_{n}^{(0)}$.
\end{lemma}

\begin{proof}
In view of Lemma \ref{L3.3}(i) we note that $p(K)=\{0\}$. It follows from Proposition \ref{P1} that $p(UT_n)\subseteq UT_n^{(0)}$. It suffices to prove that $UT_n^{(0)}\subseteq p(UT_n)$.

In view of Lemma \ref{L3.3}(iii) we have that there exist  $i_0\in\{1,\ldots,m\}$ and $\bar{b}_1,\bar{b}_2\in K^m$ such that
\[
p_{i_0}(\bar{b}_1,\bar{b}_2)\neq 0.
\]
Since $|K|\geq n$ we get from Lemma \ref{L3.5} that there exist $\bar{c}_1,\ldots,\bar{c}_n\in K^m$ such that
\[
p_{i_0}(\bar{c}_{s},\bar{c}_{t})\neq 0
\]
for all $1\leq s<t\leq n$.

For any $A'=(a_{jk}')\in UT_{n}^{(0)}$, we take $u_i=(a_{jk}^{(i)})\in UT_n$, $i=1,\ldots,m$, where
\[
\left\{
\begin{aligned}
\bar{a}_{jj}&=\bar{c}_{j}\quad\mbox{for all $j=1,\ldots,n$};\\
a_{st}^{(i_0)}&=x_{st}^{(i_0)}\quad\mbox{for all $1\leq s<t\leq n$};\\
a_{st}^{(i)}&=0,\quad\mbox{otherwise}.
\end{aligned}
\right.
\]
It follows from Proposition \ref{P1} that
\[
p(u_1,\ldots,u_m)=(p_{st}),
\]
where $p_{ss}=0$ for $s=1,\ldots,n$, and
\begin{eqnarray}\label{e6}
\begin{split}
p_{st}&=p_{i_0}(\bar{c}_{s},\bar{c}_{t})x_{st}^{(i_0)}\\
&\ \ \ +\sum\limits_{k=2}^{t-s}\left(\sum\limits_{s=j_1<j_2<\cdots <j_{k+1}=t}
p_{i_1\cdots i_k}(\bar{c}_{j_1j_1},\ldots,\bar{c}_{j_{k+1}j_{k+1}})x_{j_1j_2}^{(i_0)}\cdots x_{j_{k}j_{k+1}}^{(i_0)}\right)
\end{split}
\end{eqnarray}
for all $1\leq s<t\leq n$.  Consider the following group of equations:
\begin{eqnarray}\label{e7}
\begin{split}
a_{st}'&=p_{i_0}(\bar{c}_{s},\bar{c}_{t})x_{st}^{(i_0)}\\
&\ \ \ +\sum\limits_{k=2}^{t-s}\left(\sum\limits_{s=j_1<j_2<\cdots <j_{k+1}=t}
p_{i_1\cdots i_k}(\bar{c}_{j_1j_1},\ldots,\bar{c}_{j_{k+1}j_{k+1}})x_{j_1j_2}^{(i_0)}\cdots x_{j_{k}j_{k+1}}^{(i_0)}\right)
\end{split}
\end{eqnarray}
for all $1\leq s<t\leq n$. Since $p_{i_0}(\bar{c}_{s},\bar{c}_{t})\neq 0$ for all $1\leq s<t\leq n$, we easily check that the group of equations (\ref{e7}) has a solution
\[
\left\{c_{st}^{(i_0)}\in K~|~\mbox{for all $1\leq s<t\leq n$}\right\}.
\]
We take
\[
x_{st}^{(i_0)}=c_{st}^{(i_0)}
\]
for all $1\leq s<t\leq n$. It follows from both (\ref{e6}) and (\ref{e7}) that
\[
p(u_1,\ldots,u_m)=(p_{st})=(a_{st}')=A'.
\]
This implies that $T_n(K)^{(0)}\subseteq p(T_n(K))$ as desired. This proves the result.
\end{proof}

The following result is crucial for the proof of our main result.

\begin{lemma}\label{L3.7}
Let $p(x_1,\ldots,x_m)$ be a nonzero linear polynomial with zero constant term over a field $K$. Suppose that ord$(p)=r$, $2\leq r\leq n-2$. Suppose that $|K|>\frac{(2n-3r+1)r}{2}$.  We have that $p(UT_n)=UT_{n}^{(r-1)}$.
\end{lemma}

\begin{proof}
In view of both Proposition \ref{P1} and Lemma \ref{L3.3}(i)(ii) we note that $p(UT_n)\subseteq UT_n^{(r-1)}$. It suffices to prove that $UT_n^{(r-1)}\subseteq p(UT_n)$. In view of Lemma \ref{L3.3}(iii) we have that
\[
p_{i_1'\cdots i_r'}(K)\neq \{0\}
\]
for some $(i_1',\ldots,i_r')\in T_r$. Since $|K|>\frac{(2n-3r+1)r}{2}$ we get from Lemma \ref{L3.5} that there exist $\bar{b}_1,\ldots,\bar{b}_n\in K^m$ such that
\begin{equation}\label{e8}
p_{i_1'\cdots i_r'}(\bar{b}_{s},\ldots,\bar{b}_{r+s-1},\bar{b}_{r+s+t})\neq 0
\end{equation}
for all $1\leq s<r+s+t\leq n$. We set
\[
\widehat{b}_{s,t}=(\bar{b}_{s},\ldots,\bar{b}_{r+s-1},\bar{b}_{r+s+t})
\]
for all $1\leq s<r+s+t\leq n$. It follows from (\ref{e8}) that
\[
p_{i_1'\cdots i_r'}(\widehat{b}_{s,t})\neq 0
\]
for all $1\leq s<r+s+t\leq n$. For any $u_i=(a_{jk}^{(i)})\in UT_n$, $i=1,\ldots,m$, we take $\bar{a}_{jj}=\bar{b}_j$ for all $j=1,\ldots,n$. It follows from Proposition \ref{P1} that
\[
f(u_1,\ldots,u_m)=(p_{s,r+s+t})
\]
where

\begin{eqnarray}\label{e9}
\begin{split}
p&_{s,r+s+t}=\sum\limits_{k=r}^{r+t}\left(\sum\limits_{\substack{s=j_1<\cdots <j_{k+1}=r+s+t\\(i_1,\ldots,i_k)\in T^k_m}}p_{i_1\cdots i_k}(\bar{b}_{j_1},\ldots,\bar{b}_{j_{k+1}})a_{j_1j_2}^{(i_1)}\cdots a_{j_{k}j_{k+1}}^{(i_k)}\right)\\
&=\sum\limits_{\substack{s=j_1<\cdots <j_{r+1}=r+s+t\\(i_1,\ldots,i_r)\in T_r}}p_{i_1\cdots i_r}(\bar{b}_{j_1},\ldots,\bar{b}_{j_{r+1}})a_{j_1j_2}^{(i_1)}\cdots a_{j_{r}j_{r+1}}^{(i_r)}\\
&\ \ \ +\sum\limits_{k=r+1}^{r+t}\left(\sum\limits_{\substack{s=j_1<\cdots <j_{k+1}=r+s+t\\(i_1,\ldots,i_k)\in T^k_m}}p_{i_1\cdots i_k}(\bar{b}_{j_1},\ldots,\bar{b}_{j_{k+1}})a_{j_1j_2}^{(i_1)}\cdots a_{j_{k}j_{k+1}}^{(i_k)}\right)\\
&=\left(\sum\limits_{\substack{(i_1,\ldots,i_r)\in T_r\\i_r=i_r'}}p_{i_1\cdots i_r}(\widehat{b}_{s,t})a_{s,s+1}^{(i_1)}\cdots a_{r+s-2,r+s-1}^{(i_{r-1})}\right)a_{r+s-1,r+s+t}^{(i_r')}\\
&\ \ \ +\sum\limits_{\substack{(i_1,\ldots,i_{r})\in T_r\\ i_r\neq i_r'}}p_{i_1\cdots i_r}(\widehat{b}_{s,t})a_{s,s+1}^{(i_1)}\cdots a_{r+s-2,r+s-1}^{(i_{r-1})}a_{r+s-1,r+s+t}^{(i_r)}\\
&\ \ \ +\sum\limits_{\substack{s\leq j_1<\cdots j_{r+1}\leq r+s+t\\(j_r,j_{r+1})\neq (r+s-1,r+s+t)\\(i_1,\ldots,i_r)\in T_r}}p_{i_1\cdots i_r}(\bar{b}_{j_1},\ldots,\bar{b}_{j_{r+1}})a_{j_1j_2}^{(i_1)}\cdots a_{j_rj_{r+1}}^{(i_r)}\\
&\ \ \ +\sum\limits_{k=r+1}^{t-s}\left(\sum\limits_{\substack{s=j_1<j_2<\cdots <j_{k+1}=r+s+t\\(i_1,\ldots,i_k)\in T^k_m}}
p_{i_1\cdots i_k}(\bar{b}_{j_1},\ldots,\bar{b}_{j_{k+1}})a_{j_1j_2}^{(i_1)}\cdots a_{j_{k}j_{k+1}}^{(i_k)}\right)
\end{split}
\end{eqnarray}

for all $1\leq s<r+s+t\leq n$.  We set
\[
h_{s,r+s+t}=\sum\limits_{\substack{(i_1,\ldots,i_{r})\in T_r\\ i_r=i_r'}}p_{i_1\cdots i_r}(\widehat{b}_{s,t})x_{s,s+1}^{(i_1)}\cdots x_{r+s-2,r+s-1}^{(i_{r-1})}
\]
for all $1\leq s<r+s+t\leq n$. By $W_{s,r+s+t}$ we denote the index set of the variables in $h_{s,r+s+t}$, where $1\leq s<r+s+t\leq n$. That is, $h_{s,r+s+t}$ is a linear polynomial on the variables
\[
\left\{x_{jk}^{(i)}\in X~|~(j,k,i)\in W_{s,r+s+t}\right\}
\]
for all $1\leq s<r+s+t\leq n$. We claim that
\[
h_{s,r+s+t}(K)\neq\{ 0\}
\]
for all $1\leq s<r+s+t\leq n$. Indeed, we take
\[
\left\{
\begin{aligned}
a_{j,j+1}^{(i_{j-s+1}')}&=1\quad\mbox{for all $j=s,\ldots,r+s-2$};\\
a_{jk}^{(i)}&=0\quad\mbox{otherwise}.
\end{aligned}
\right.
\]
It follows that
\begin{eqnarray*}
\begin{split}
h_{s,r+s+t}(a_{jk}^{(i)})&=p_{i_1'\cdots i_r'}(\widehat{b}_{s,t})a_{s,s+1}^{(i_1')}\cdots a_{r+s-2,r+s-1}^{(i_{r-1}')}\\
&=p_{i_1'\cdots i_r'}(\widehat{b}_{s,t})\neq 0,
\end{split}
\end{eqnarray*}
as desired. Set
\[
W=\bigcup_{1\leq s<r+s+t\leq n}W_{s,r+s+t}.
\]
Since $|K|>\frac{(2n-3r+1)r}{2}$ we get from Corollary \ref{CC1} that there exist $b_{jk}^{(i)}\in K$, where $(j,k,i)\in W$ such that
\[
h_{s,r+s+t}(b_{jk}^{(i)})\neq 0
\]
for all $1\leq s<r+s+t\leq n$. We define an order of the set of variables
\[
\left\{x_{r+s-1,r+s+t}^{(i_r')}\in X~|~1\leq s<r+s+t\leq n\right\}
\]
as follows:
\[
x_{r,r+1}^{(i_r')}<\cdots <x_{n-1,n}^{(i_r')}<x_{r,r+2}^{(i_r')}<\cdots <x_{n-2,n}^{(i_r')}<\cdots <x_{r,n}^{(i_r')}.
\]

We take
\[
\left\{
\begin{aligned}
a_{jk}^{(i)}&=b_{jk}^{(i)}\quad\mbox{for all $(j,k,i)\in W$};\\
a_{r+s-1,r+s+t}^{(i_r')}&=x_{r+s-1,r+s+t}^{(i_r')}\quad\mbox{for all $1\leq s<r+s+t\leq n$};\\
a_{jk}^{(i)}&=0,\quad\mbox{otherwise}
\end{aligned}
\right.
\]
in (\ref{e9}). It follows from (\ref{e9}) that
\[
p_{1,r+1}=h_{1,r+1}(b_{jk}^{(i)})x_{r,r+1}^{(i_r')}+\alpha_{1,r+1},
\]
where $\alpha_{1,r+1}\in K$ and
\[
p_{2,r+2}=h_{2,r+2}(b_{jk}^{(i)})x_{r+1,r+2}^{(i_r')}+\alpha_{2,r+2}(x_{r,r+1}^{(i_r')}),
\]
where $\alpha_{2,r+2}$ is a polynomial with the variable $x_{r,r+1}^{(i_r')}$ over $K$. Continuing the same arguments as above we can get from (\ref{e9}) that
\begin{equation}\label{e10}
p_{s,r+s+t}=h_{s,r+s+t}(b_{jk}^{(i)})x_{r+s-1,r+s+t}^{(i_r')}+\alpha_{s,r+s+t}(x_{1,r+1}^{(i_r')},\ldots, x_{s_1,t_1}^{(i_r')})
\end{equation}
for all $1\leq s\leq r+s+t\leq n$, where $\alpha_{s,r+s+t}$ is a polynomial with all previous variables $\{x_{1,r+1}^{(i_r')},\ldots, x_{s_1,t_1}^{(i_r')}\}$ of $x_{r+s-1,r+s+t}^{(i_r')}$ over $K$.

For any $A=(a_{jk}')\in UT_n^{(r-1)}$, we consider the following the group of equations:
\begin{equation}\label{e11}
h_{s,r+s+t}(b_{jk}^{(i)})x_{r+s-1,r+s+t}^{(i_r')}+\alpha_{s,r+s+t}(x_{1,r+1}^{(i_r')},\ldots, x_{s_1,t_1}^{(i_r')})=a_{s,r+s+t}'
\end{equation}
for all $1\leq s<r+s+t\leq n$. Since $\alpha_{1,r+1}\in K$ and
\[
h_{s,r+s+t}(b_{jk}^{(i)})\neq 0
\]
for all $1\leq s<r+s+t\leq n$, we easily check that the group of equations (\ref{e11}) has a solution
\[
\left\{b_{r+s-1,r+s+t}^{(i_r')}\in K~|~1\leq s<r+s+t\leq n\right\}.
\]

Finally we take
\[
x_{r+s-1,r+s+t}^{(i_r')}=b_{r+s-1,r+s+t}^{(i_r')}
\]
for all $1\leq s<r+s+t\leq n$ in (\ref{e10}). It follows from both (\ref{e10}) and (\ref{e11}) that
\[
p_{s,r+s+t}=a_{s,r+s+t}'
\]
for all $1\leq s<r+s+t\leq n$. This implies that

\[
f(u_1,\ldots,u_m)=\left(p_{s,r+s+t}\right)=\left(a_{s,r+s+t}'\right)=A'.
\]
We obtain that $UT_n^{(r-1)}\subseteq p(UT_n)$ as desired. The proof of the result is complete.
\end{proof}

The following result is similar to \cite[Lemma 3.12]{ChenLuoWang1}. We give its proof for completeness.

\begin{lemma}\label{L3.8}
Let $p(x_1,\ldots,x_m)$ be a linear polynomial with zero constant term over a field $K$. Suppose that \emph{ord}$(p)=n-1$. We have that $p(T_n(K))=T_{n}(K)^{(n-2)}$.
\end{lemma}

\begin{proof}
In view of Proposition \ref{P1} and Lemma \ref{L3.2}(ii) we have that $p(T_n(K))\subseteq T_n(K)^{(n-2)}$. It suffices to prove that $T_n(K)^{(n-2)}\subseteq p(T_n(K))$.

Since ord$(p)=n-1$ we get that $p(T_n(K))\neq\{0\}$. It implies from Proposition \ref{P1} that
\[
p_{1n}(K)\neq \{0\}.
\]
Note that $p_{1n}$ is a linear polynomial with zero constant term over $K$. We get from Lemma \ref{L3.1} that
\[
p_{1n}(K)=K.
\]

For any $a_{1n}'\in K$ we get that that there exist $a_{jk}^{(i)}\in K$, where $1\leq j\leq k\leq n$, $i=1,\ldots,m$, such that
\[
p_{1n}(a_{jk}^{(i)})=a_{1n}'.
\]
We set $u_{i}=(a_{jk}^{(i)})\in T_n(K)$, $i=1,\ldots,m$. It follows from (\ref{e2}) that
\begin{eqnarray*}
\begin{split}
p(u_1,\ldots,u_m)&=\left(
\begin{array}{cccc}
0 & 0 & \ldots & p_{1n}(a_{jk}^{(i)})\\
0 & 0 & \ldots &  0 \\
 \vdots & \vdots & \ddots & \vdots\\
 0 & 0 & \ldots & 0
\end{array} \right)\\
&=\left(
\begin{array}{cccc}
0 & 0 & \ldots & a_{1n}'\\
0 & 0 & \ldots &  0 \\
 \vdots & \vdots & \ddots & \vdots\\
 0 & 0 & \ldots & 0
\end{array} \right).
\end{split}
\end{eqnarray*}
This implies that $T_n(K)^{(n-2)}\subseteq p(T_n(K))$ as desired. This proves the result.
\end{proof}

We are ready to give the proof of the main result of the section.

\begin{theorem}\label{T3}
Let $n\geq 2$ and $m\geq 1$ be integers. Let $K$ be a field, let $p(x_1,\ldots, x_m)$ be a nonzero linear polynomial with zero constant term over $K$. Set $r=ord(p)$. Then one of the following statement is true:
\begin{enumerate}
\item[(i)] Suppose that $r=0$. We have that $p(UT_n)=UT_n$;
\item[(ii)] Suppose that $r=1$ and $|K|\geq n$. We have that $p(UT_n)=UT_{n}^{(0)}$;
\item[(iii)] Suppose that $2\leq r\leq n-2$ and $|K|>\frac{(2n-3r+1)r}{2}$. We have that $p(UT_n)=UT_{n}^{(r-1)}$;
\item[(iv)] Suppose that $r=n-1$. We have that $p(UT_n)=UT_{n}^{(n-2)}$;
\item[(v)] Suppose that $r\geq n$. We have that $p(UT_{n})=\{0\}$.
\end{enumerate}
\end{theorem}

\begin{proof}
The statement (i) follows from Lemma \ref{L3.2}. The statement (ii) follows from Lemma \ref{L3.6}. The statement (iii) follows from Lemma \ref{L3.7}. The statement (iv) follows from Lemma \ref{L3.8}. The last statement follows from the definition of ord$(p)$.
\end{proof}

The following useful result is trivial.

\begin{lemma}\label{L3.9}
Let $n\geq 2$ be an integer, let $0\leq r\leq n-1$ be an integer. We have that
\begin{enumerate}
\item[(i)] Suppose that $n=2$. We have that
\[
\frac{(2n-3r+1)r}{2}\leq 1
\]
for all $0\leq r\leq n-1$;
\item[(ii)] Suppose that $n\geq 3$. We have that
\[
\frac{n(n-1)}{3}\geq \frac{(2n-3r+1)r}{2}
\]
for all $0\leq r\leq n-1$.
\end{enumerate}
\end{lemma}

We are in a position to give the proof of Theorem \ref{T2}.

\begin{proof}[The proof of Theorem \ref{T2}]
Set $r=ord(p)$. We note that $0\leq r\leq \frac{m}{2}$. Suppose first that $n=2$. In view of Lemma \ref{L3.8}(i) we note that $|K|>\frac{(2n-3r+1)r}{2}$ for all $0\leq r\leq n-1$. Suppose next that $n\geq 3$. In view of Lemma \ref{L3.8}(ii) we note that $|K|>\frac{(2n-3r+1)r}{2}$ for all $0\leq r\leq n-1$. Then the result follows from Theorem \ref{T3}.
\end{proof}

\section{Applications}

By $UT_n^{(-)}$ we denote the Lie algebra defined on $UT_n$ by means of the Lie bracket
\[
[a,b]=ab-ba.
\]

Recently, Fagundes and Koshlukov \cite{FK} obtained the following result:

\begin{theorem}\cite[Corollary 2.9]{FK}\label{T4.1}
Let $F$ be a field with $|K|\geq\frac{n(n-1)}{2}$, let $f$ be a multilinear Lie polynomial. Then $Im(f)$ on $UT^{(-)})$ is $J^r$, for some $0\leq r\leq n$.
\end{theorem}

Applying Theorem \ref{T2} and using the same arguments as in \cite[Corollary 2.9]{FK}, we can obtain the following result. We omit its proof for brevity.

\begin{corollary}
Let $F$ be a field with $|K|>\frac{n(n-1)}{3}$, let $f$ be a multilinear Lie polynomial. Then $Im(f)$ on $UT^{(-)})$ is $J^r$, for some $0\leq r\leq n$.
\end{corollary}

Recently, Fagundes and Koshlukov \cite{FK} obtained the following result:

\begin{theorem}\cite[Theorem 4.4]{FK}
Let $F$ be a field with $|K|\geq\frac{n(n-1)}{2}$, let $UT_n =\bigoplus_{k\in Z_q}A_k$  be endowed with
the elementary $Z_q$-grading given by the sequence
\[
(\bar{0},\bar{1},\ldots,\overline{q-2},\overline{q-1},\ldots,\overline{q-1})
\]
and let $f\in F\langle X\rangle^{gr}$ be
a multilinear polynomial. Then $Im(f)$ on $UT_n$ is $\{0\}$, $J^r$, $B_{\bar{l},r}$, or $A_{\bar{l}}$, where $J=Jac(A_0)$. In particular, the image is always a homogeneous vector subspace.
\end{theorem}

Applying Theorem \ref{T2} and using the  same arguments as in \cite[Theorem 4.4]{FK}, we can obtain the following result. We omit its proof for brevity.

\begin{corollary}
Let $F$ be a field with $|K|>\frac{n(n-1)}{3}$, let $UT_n =\bigoplus_{k\in Z_q}A_k$  be endowed with
the elementary $Z_q$-grading given by the sequence
\[
(\bar{0},\bar{1},\ldots,\overline{q-2},\overline{q-1},\ldots,\overline{q-1})
\]
and let $f\in F\langle X\rangle^{gr}$ be
a multilinear polynomial. Then $Im(f)$ on $UT_n$ is $\{0\}$, $J^r$, $B_{\bar{l},r}$, or $A_{\bar{l}}$, where $J=Jac(A_0)$. In particular, the image is always a homogeneous vector subspace.
\end{corollary}

\end{document}